\documentclass[11pt,reqno]{amsart}
\usepackage[utf8]{inputenc}
\usepackage{a4wide}
\usepackage{appendix}

\usepackage{graphicx}
\usepackage{enumerate}
\usepackage{hyperref}
\usepackage{cleveref}
\usepackage{amsmath,amsopn,amssymb,amsfonts,stmaryrd}
\usepackage{verbatim}
\usepackage{tikz-cd}
\usepackage{amsthm}
\usepackage{mathtools}
\usepackage{color}
\usepackage{enumitem}
\usepackage[framemethod=TikZ]{mdframed}
\usepackage{bbm}
\usepackage{mathrsfs}
\usepackage{booktabs}
\usepackage{caption}
\usepackage{bm}
\usepackage{tensor}

\makeatletter
\@namedef{subjclassname@2020}{%
	\textup{2020} Mathematics Subject Classification}
\makeatother

\theoremstyle{plain}
\newtheorem{thm}{Theorem}[section]
\newtheorem{cor}[thm]{Corollary}
\newtheorem{lem}[thm]{Lemma}
\newtheorem{prop}[thm]{Proposition}
\newtheorem{conj}[thm]{Conjecture}
\newtheorem{ques}[thm]{Question}
\newtheorem{fact}[thm]{Fact}
\newtheorem{facts}[thm]{Facts}
\newtheorem{rem}[thm]{Remark}

\theoremstyle{definition}
\newtheorem{example}[thm]{Example}
\newtheorem{defn}[thm]{Definition}

\usepackage{aliascnt}
\usepackage{cleveref}

\theoremstyle{plain} 
\newaliascnt{cor}{thm}
\newtheorem{cor}[cor]{Corollary}
\aliascntresetthe{cor}
\crefname{cor}{corollary}{corollaries}
\Crefname{cor}{Corollary}{Corollaries}

\newaliascnt{lem}{thm}
\newtheorem{lem}[lem]{Lemma}
\aliascntresetthe{lem}
\crefname{lem}{lemma}{lemmas}
\Crefname{lem}{Lemma}{Lemmas}

\newaliascnt{prop}{thm}
\newtheorem{prop}[prop]{Proposition}
\aliascntresetthe{prop}
\crefname{prop}{proposition}{propositions}
\Crefname{prop}{Proposition}{Propositions}

\newaliascnt{conj}{thm}
\newtheorem{conj}[conj]{Conjecture}
\aliascntresetthe{conj}
\crefname{conj}{conjecture}{conjectures}
\Crefname{conj}{Conjecture}{Conjectures}

\crefname{ques}{question}{questions}
\Crefname{ques}{Question}{Questions}

\newaliascnt{fact}{thm}
\newtheorem{fact}[fact]{Fact}
\aliascntresetthe{fact}
\crefname{fact}{fact}{facts}
\Crefname{fact}{Fact}{Facts}

\newaliascnt{facts}{thm}
\newtheorem{facts}[facts]{Facts}
\aliascntresetthe{facts}
\crefname{facts}{fact}{facts}
\Crefname{facts}{Fact}{Facts}

\newaliascnt{rem}{thm}
\newtheorem{rem}[rem]{Remark}
\aliascntresetthe{rem}
\crefname{rem}{remark}{remarks}
\Crefname{rem}{Remark}{Remarks}

\theoremstyle{definition} 
\newaliascnt{example}{thm}
\newtheorem{example}[example]{Example}
\aliascntresetthe{example}
\crefname{example}{example}{examples}
\Crefname{example}{Example}{Examples}

\newaliascnt{defn}{thm}
\newtheorem{defn}[defn]{Definition}
\aliascntresetthe{defn}
\crefname{defn}{definition}{definitions}
\Crefname{defn}{Definition}{Definitions}

\crefname{clai}{claim}{claims}
\Crefname{clai}{Claim}{Claims}

\theoremstyle{definition} 

\crefname{Idef}{definition}{definitions}
\Crefname{Idef}{Definition}{Definitions}

\numberwithin{equation}{section}

\def\G{\mathcal{G}}

\def\d{\delta}

\def\k{\kappa}

\def\s{\sigma}

\def\d{\delta}

\def\k{\kappa}

\def\s{\sigma}

\def\RR{{\mathbb R}}

\def\d{{\mathrm{d}}}

\def\id{\mathrm{id}}

\def\Diff{\mathrm{Diff}}

\newcommand{\gc}{\gamma}

\setlist[itemize]{noitemsep, topsep=0pt}
\usepackage[dvipsnames]{xcolor}

\makeatletter
\newcommand{\vast}{\bBigg@{2}}
\newcommand{\Vast}{\bBigg@{5}}
\makeatother

\newcommand{\RNum}[1]{\uppercase\expandafter{\romannumeral #1\relax}}
\allowdisplaybreaks

\title[ Magnetic Geodesics on Half Lie-Groups]{The Hopf--Rinow Theorem and Mañé's Critical Value\\ for Magnetic Geodesics on Half Lie-Groups}

\author{Levin Maier}
\address{Faculty of Mathematics and Computer Science,
	University of Heidelberg,
	Im Neuenheimer Field 205,
	69120 Heidelberg, Germany}
\email{lmaier@mathi.uni-heidelberg.de}

\author{Francesco Ruscelli}
\address{Faculty of Mathematics and Computer Science,
	University of Heidelberg,
	Im Neuenheimer Field 205,
	69120 Heidelberg, Germany}
\email{fruscelli@mathi.uni-heidelberg.de}


\subjclass[2020]{}

\begin{document}

\renewcommand{\abstractname}{Abstract}
	\begin{abstract}
In this article, we investigate \emph{right-invariant magnetic systems} on half-Lie groups, which consist of a strong right-invariant Riemannian metric and a right-invariant closed two-form. The main examples are groups of \( H^s \) or \( C^k \) diffeomorphisms of compact manifolds. \\[0.3em]
In this setting, we define \emph{Mañé's critical value} on the universal cover for weakly exact right-invariant magnetic fields. First, we prove that the lift of the magnetic flow to the universal cover coincides with a Finsler geodesic flow for energies above this threshold. \\[0.3em]
Finally, we show that for energies above Mañé's critical value, the full Hopf–Rinow theorem holds for such magnetic systems, thereby generalizing the work of Contreras and Merry from closed finite-dimensional manifolds to this infinite-dimensional context. Our work extends the recent results of Bauer, Harms, and Michor from geodesic flows to magnetic geodesic flows.
	\end{abstract}
	
    \maketitle
	
	\section{Introduction}
\noindent\textbf{Half-Lie groups} are smooth manifolds and topological groups for which the right translations are smooth, while the left translations are merely continuous. Examples in which the left translations are only continuous exist only in infinite dimensions.

Their study is motivated by the fact that these equations arise naturally in the context of 
V.~Arnold's geometric formulation of mathematical hydrodynamics~\cite{Arnold66}. 
More precisely, Arnold showed that the Euler equations of hydrodynamics, 
which govern the motion of an incompressible and inviscid fluid in a fixed domain (with or without boundary), 
can be interpreted as the geodesic equations on the group of volume-preserving diffeomorphisms 
endowed with a right-invariant Riemannian metric. 
Since then, many partial differential equations arising in mathematical physics 
have been reinterpreted within a similar geometric framework~\cite{AK98, Khesin-Mis-Mod-inf-Newton, Vi08}. 

However, the group of volume-preserving diffeomorphisms is a tame Fréchet--Lie group, 
which leads to significant analytical difficulties. 
Following the approach of Ebin and Marsden~\cite{EM70}, 
one can instead work with \( C^k \)- or \( H^s \)-diffeomorphism groups, 
as first studied by Eells, Eliasson, and Palais~\cite{Eells66, Eliasson67, Palais68}, 
which constitute the main examples of half-Lie groups.

This insight has motivated a significant amount of work devoted to studying the theoretical properties of infinite-dimensional geometric spaces. Before going into detail, we would like to point out that several unexpected phenomena occur in this infinite-dimensional setting, where well-known results of finite-dimensional Riemannian geometry cease to hold. Notable examples include the non-existence of Christoffel symbols~\cite{BauerBruverisMichor2014} and the vanishing of the geodesic distance~\cite{BauerHarmsPreston2020, MichorMumford2005, JerrardMaor2019}, a phenomenon first observed by Eliashberg and Polterovich~\cite{EliashbergPolterovich1993}. In what follows, we focus on the Hopf–Rinow theorem and begin by briefly recalling it.\\
 
\noindent\textbf{The Hopf–Rinow theorem}~\cite{HR}, which is one of the central results in Riemannian geometry on finite-dimensional manifolds, states that the notions of \emph{geodesic completeness} and \emph{metric completeness} coincide. Moreover, it implies \emph{geodesic convexity}, that is, any two points can be connected by a geodesic. 

In infinite dimensions, however, the situation is far more delicate, as in general all these implications may fail; see~\cite{HopfrinowfalseAktkin, HopfrinowfalseEkeland, Atkin97}. Surprisingly, due to a recent result of Bauer, Harms, and Michor~\cite[Thm.~7.7]{Bauer_2025}, the full Hopf–Rinow theorem holds for half-Lie groups equipped with right invariant strong Riemannian metrics. \\
The main objective of this article is to extend this result to \emph{right-invariant magnetic systems} on half-Lie groups, a framework pioneered in finite dimensions by V.~Arnold~\cite{ar61}.\\ 

\paragraph{\textbf{Right-invariant magnetic systems on half-Lie groups.}} In his pioneering work~\cite{ar61}, V.~Arnold described the motion of a charged particle in a magnetic field on a finite-dimensional phase space in the language of modern dynamical systems. 
The same mathematical ideas also extend naturally---though with additional analytical challenges---to systems whose phase space is infinite-dimensional. \\
We now formulate this framework within the setting of half-Lie groups $G$ equipped with a strong $G$-right-invariant metric~$\mathcal{G}$ (for a precise definition of this notion, see \Cref{Def: strong right invariant Riemannian metrics}).\\
A closed two-form $\sigma \in \Omega^2(G)$ is called a \emph{magnetic field}, and the triple $(G, \mathcal{G}, \sigma)$ is referred to as a \emph{magnetic system}. 
In the following, we assume that $\sigma$ is $G$-right-invariant as well, and we then call $(G, \mathcal{G}, \sigma)$ a \emph{$G$-right-invariant magnetic system}.

This triple determines a unique skew-symmetric bundle endomorphism $Y \colon TG \to TG$, called the \emph{Lorentz force}, defined by
\begin{equation}\label{e:Lorentz}
    \mathcal{G}_x\!\left(Y_x u, v\right) = \sigma_x(u,v),
    \qquad \forall\, x \in G,\ \forall\, u,v \in T_xG.
\end{equation}
By \cite[Thm.~7.1]{Bauer_2025}, the pair $(G, \mathcal{G})$ is a Hilbert manifold equipped with a strong Riemannian metric. Combining \eqref{e:Lorentz} with the Riesz representation theorem shows that $Y$ exists and is unique. Moreover, $Y$ is $G$-right-\textit{equivariant}, since both $\mathcal{G}$ and $\sigma$ are $G$-right-invariant in \eqref{e:Lorentz}. \\
We call a smooth curve $\gamma : I\subseteq\mathbb{R} \to G$ a \emph{magnetic geodesic} of $(G, \mathcal{G}, \sigma)$ if it satisfies
\begin{equation}\label{e:mg}
    \nabla_{\dot{\gamma}} \dot{\gamma} = Y_{\gamma} \dot{\gamma},
\end{equation}
where $\nabla$ denotes the Levi–Civita connection of the metric $\mathcal{G}$.  From \eqref{e:mg}, we see that a magnetic geodesic with $\sigma = 0$ is simply a standard geodesic for the metric $\mathcal{G}$. \\
Since $Y$ is skew-symmetric, magnetic geodesics have constant kinetic energy $E(\gamma, \dot{\gamma}) := \tfrac{1}{2}\, \mathcal{G}_{\gamma}(\dot{\gamma}, \dot{\gamma})$
and hence constant speed $|\dot{\gamma}| := \sqrt{\mathcal{G}_{\gamma}(\dot{\gamma}, \dot{\gamma})}$ just like standard geodesics. 
\Cref{s:preliminaries} contains a thorough exposition of inifinite-dimensional magnetic systems.

However, unlike standard geodesics, magnetic geodesics cannot, in general, be reparametrized to have unit speed. This is because the left-hand side of \eqref{e:mg} scales quadratically with the speed, whereas the right-hand side scales only linearly. Therefore, one of the main points of interest is to understand the similarities and differences between standard and magnetic geodesics.\\
On finite-dimensional closed manifolds, an important role is played by \emph{Mañé's critical value's}~\cite{CIPP, Man}, which marks an energy threshold indicating significant dynamical and geometric transitions in the magnetic geodesic flow. Next, we extend this notion to our infinite-dimensional setting. \\

\paragraph{\textbf{Mañé's critical value on the universal cover.}}  
A two-form \(\sigma\) on \(G\) is called \textit{weakly exact} if \( \sigma|_{\pi_2(G)} = 0 \), that is if the pullback \( \tilde{\sigma} \) of \(\sigma\) to the universal cover $\hat{G}$ of $G$ is exact.
We denote by \(\hat{\G}\) the pullback of \(\G\) to \(\hat{G}\), which is a strong right-invariant Riemannian metric. 
Details of this construction are provided in \Cref{s: Mañé's critical value on the universal cover} for the convenience of the reader.  

\begin{defn}[=\Cref{Def: int maneuni half}]\label{Def: int maneuni half int}
Let \((G, \mathcal{G}, \sigma)\)  be a \( G \)-right-invariant magnetic system.  If $\sigma$ is weakly exact we define \emph{Mañé's critical value} of the universal cover of the magnetic system \((G, \mathcal{G}, \sigma)\) by
\begin{equation}\label{e:maneuni_half_int}
c(G, \mathcal{G}, \sigma) := \inf_{\mathrm{d}\hat{\alpha} = \hat{\sigma}} \frac{1}{2} \,\Vert \hat{\alpha} \Vert_\infty^2 := \inf_{\mathrm{d}\hat{\alpha} = \hat{\sigma}} \frac{1}{2} \sup_{\hat x\in \hat{G}}\,\vert \hat{\alpha}_{\hat x} \vert_{\hat{\G}}^2
\;\in [0,\infty],
\end{equation}
where the infimum is taken over all $\hat{G}$-right-invariant primitives $\hat{\alpha}$ of $\hat{\sigma}$ (and it is understood to be \( \infty \) should such a primitive not exist). Here, \( |\cdot |_   \mathcal{\hat{G}} \) If \( \sigma \) is not weakly exact, we set \( c(G, g, \sigma) = \infty \).
\end{defn}
\begin{rem}
Note that Mañé's critical value on the universal cover for the standard magnetic system \( (\mathbb{R}^2, g_{\mathrm{eucl}}, \omega_{\mathrm{std}}) \), where \( g_{\mathrm{eucl}} \) is the Euclidean metric and \( \omega_{\mathrm{std}} \) the standard symplectic form, is infinite, since the standard symplectic form \( \omega_{\mathrm{std}} \) does not admit an \( \mathbb{R}^2 \)-invariant primitive.
\end{rem}
\begin{rem}
    In contrast to the finite-dimensional case considered in~\cite{CFP10}, we require that the primitive of the magnetic field on the universal cover be \( \hat{G} \)-right-invariant. This assumption is most likely necessary for establishing a Hopf–Rinow-type theorem in this context. For instance, without right invariance of the metric, the result does not hold for geodesic flows, as \cite{HopfrinowfalseAktkin} shows.
\end{rem}
Before moving on, we point out that if $\sigma$ is weakly exact and admits an $\hat{G}$-right-invariant primitive $\hat{\alpha}$, then, due to the $\hat{G}$-right-invariance, we obtain the following upper bound:
\[
    c(G, \mathcal{G}, \sigma) \le \tfrac{1}{2}\, |\hat{\alpha}_e|_{\hat{\mathcal{G}}}^2.
\]
Hence, Mañé critical value $c(G, \mathcal{G}, \sigma)$ is finite in this case.  
In the next section, we illustrate \Cref{Def: int maneuni half int} by showing that Mañé's critical value of the universal cover marks an energy threshold indicating significant dynamical and geometric transitions in the magnetic geodesic flow.

\subsection{Main results.}
The central goal of this work is to extend the Hopf–Rinow theorem~\cite[Thm.~7.7]{Bauer_2025} to the setting of right-invariant magnetic systems on half-Lie groups.  
As a first step, we show that the Mañé critical value of the universal cover defines an energy threshold distinguishing different dynamical and geometric regimes of the magnetic geodesic flow, as stated in the following result and in \Cref{IThm: HopfRinow half lie group magnetic int}:
\begin{prop}\label{prop: conjugated to Finsler flow}
    Let $(G, \mathcal{G}, \sigma)$ be a $G$-right-invariant magnetic system with $\sigma$ weakly exact and assume its pullback $\hat{\sigma}$ to the universal cover $\hat{G}$ admits a $\hat{G}$-right-invariant primitive \( \hat{\alpha} \).
    
    Then, for every \( \k > c(G, \mathcal{G}, \sigma) \), we have
    \begin{equation*}\label{e:finsler_metric}
        \mathcal{F}^{\kappa} \colon T\hat{G} \to \mathbb{R}, 
        \qquad 
        \mathcal{F}^{\kappa}(\hat{x}, \hat{v}) := \sqrt{2k}\, |\hat{v}|_{\hat{x}} - \hat{\alpha}_{\hat{x}}(\hat{v}),
        \quad 
        \forall\, (\hat{x}, \hat{v}) \in T\hat{G},
    \end{equation*}
    defines a $\hat{G}$-right-invariant Finsler metric on \( \hat{G} \) satisfying the following properties:
    \begin{enumerate}
        \item There exist constants \( C_1, C_2 > 0 \) such that
        \[
            C_1\, \mathcal{F}^\kappa_{e}(\hat{v})
            \;\le\; \sqrt{\hat{\mathcal{G}}_{e}(\hat{v},\hat{v})}
            \;\le\; C_2\, \mathcal{F}^\kappa_{e}(v),
            \qquad 
            \forall\, \hat{v} \in T_e\hat{G} \setminus \{0\}.
        \]
       \item On the energy surface \( \hat{E}^{-1}(\kappa) := \Sigma_{\kappa} \subseteq T\hat{G} \), as defined in~\eqref{eq: energy surface}, 
the lift of the magnetic flow of \( (G, \mathcal{G}, \sigma) \) 
to the universal cover \( \hat{G} \) coincides with the Finsler geodesic flow of \( (\hat{G}, \mathcal{F}^{\kappa}) \).
    \end{enumerate}
\end{prop}
\begin{rem}
  This result generalizes~\cite[Cor.~2]{CIPP} from closed finite-dimensional manifolds to the infinite-dimensional setting. 
However, we would like to point out that, in extending the original argument from~\cite{CIPP} to our setting, the following obstacles arise: 
it is not clear whether the Finsler metric constructed there is $\hat{G}$-right-invariant, 
nor whether it induces the same topology on $\hat{G}$ as $\hat{\G}$. 
Both of these properties are essential for the proof of \Cref{IThm: HopfRinow half lie group magnetic int}. 
It would be interesting to determine whether this construction can also be carried out in the infinite-dimensional setting. 
We conclude by noting that our proof relies on variational methods.
\end{rem}
This result will play a crucial role in establishing the main theoretical advance of this article, namely the Hopf–Rinow theorem for magnetic geodesics on half-Lie groups.  
To state this properly, we first fix some notation. For a magnetic system \((G, \mathcal{G}, \sigma)\) as in \Cref{prop: conjugated to Finsler flow}, we denote by \(\hat{d}_{\mathcal{F}^\kappa}\) the geodesic distance associated with the Finsler metric \(\mathcal{F}^\kappa\) defined there.  

Analogously to the definition of the Riemannian exponential map, one can define the \emph{magnetic exponential map} \(\exp^{\mathcal{G}, \sigma}\)  of the magnetic system \((G, \mathcal{G}, \sigma)\). For details and a precise definition, we refer to~\eqref{Magnetic exponential map}. 

We are now in a position to state the main theorem.
\begin{thm}[=\Cref{IThm: HopfRinow half lie group magnetic}]\label{IThm: HopfRinow half lie group magnetic int}
  Let $(G,\G,\s)$ be as in \Cref{prop: conjugated to Finsler flow}. Then the following holds true
\begin{enumerate}
    \item \label{it: main thm a int} For all energy levels \( \kappa > c(G,\G,\sigma) \), with \( c(G,\G,\sigma )\) as in~\eqref{e:maneuni_half}, the space \( (\hat G, \hat d_{\mathcal{F}^\kappa}) \) is a complete metric space, i.e., every \( \hat d_{\mathcal{F}^\kappa}\)-Cauchy sequence converges in \( \hat G \).
    
    \item \label{it: main thm b int} The magnetic exponential map \( \exp^{\G,\sigma}_e : T_e G \to G \) is defined on all of \( T_e G \).
    
    \item \label{it: main thm c int} The magnetic exponential map \( \exp^{\G,\sigma} : TG \to G \) is defined on all of \( TG \).
    
    \item \label{it: main thm d int} The magnetic system \( (G, \G, \sigma) \) is magnetically geodesically complete, i.e.\ every magnetic geodesic is maximally defined on all of \( \mathbb{R} \).
\end{enumerate}
Assume in addition that \( G \) is \( L^2 \)-regular and that for each \( x \in  G \), the sets
\[
A_{x} := \left\{ \xi \in L^2([0,1], T_e  G) : \mathrm{evol}(\xi) = x \right\} \subset L^2([0,1], T_e  G)
\]
are weakly closed, where $\operatorname{evol}(\xi)$ is defined in \Cref{def:regular half Lie group}. Then:
\begin{enumerate}
    \setcounter{enumi}{4} 
    \item \label{it: main thm e int} For all energy levels \( \kappa > c(G, \G, \sigma) \) and all \( x, y \in G \), there exists a magnetic geodesic of \( (G, \G, \sigma) \) of energy \( \kappa \) 
 connecting \( x \) and \( y \) that minimizes the action.
\end{enumerate}
In addition, the magnetic geodesic completeness statements in items~\ref{it: main thm b int}–\ref{it: main thm d int} also hold for the magnetic systems \( (G^\ell, \mathcal{G}, \sigma) \) for all \( \ell \ge 1 \) on the weak Riemannian manifolds \( (G^\ell, \mathcal{G}) \).  
Here, \( \mathcal{G} \) and \( \sigma \) denote the restrictions of the Riemannian metric and the magnetic two-form from \( G \) to \( G^\ell \), where \( G^\ell \) denotes the space of \( C^\ell \)-elements in \( G \) in the sense of \Cref{Def: ck elements}.
\end{thm}
\begin{rem}
    For $\sigma = 0$, we recover the recent result of Bauer–Harms–Michor~\cite[Thm.~7.7]{Bauer_2025}. 
\end{rem}

\begin{rem}
   By an argument along the lines of the proof of \Cref{IThm: HopfRinow half lie group magnetic}, 
one sees that statements~\ref{it: main thm b int}--\ref{it: main thm d} hold in general 
for all right-invariant magnetic systems \( (G, \mathcal{G}, \sigma) \), regardless of whether $\sigma$ is weakly exact, and they also hold for the magnetic systems \( (G^\ell, \G, \sigma) \) for all $\ell\geq 1$.
\end{rem}
\begin{rem}
  To the best of the author's knowledge, \Cref{IThm: HopfRinow half lie group magnetic int} 
is the first result of this kind for magnetic systems on 
infinite-dimensional manifolds.  
    It extends the results of~\cite[Cor.~B]{Co06} 
for exact magnetic systems and of~\cite[Thm.~3.2]{Merry2010} 
for weakly exact magnetic systems 
from closed finite-dimensional manifolds 
to the infinite-dimensional setting.
.
\end{rem}
\begin{rem}
    If \( \k \le c(M, g, \sigma) \), there may exist pairs of points that cannot be connected by a magnetic geodesic of energy \( \k \); see~\cite{M24}. 
    This phenomenon already occurs in the finite-dimensional case, for instance when \( S^3 = SU(2) \) \cite{ALBERS2025105521}. 
    However, even for finite-dimensional closed manifolds, there is in general no good understanding of which pairs of points fail to be connected for a given subcritical energy.
\end{rem}
Next, we illustrate \Cref{IThm: HopfRinow half lie group magnetic int} in the case of the group of Sobolev diffeomorphisms on a closed manifold.

\subsection{First illustrations of \Cref{IThm: HopfRinow half lie group magnetic int}} In the following, we illustrate \Cref{IThm: HopfRinow half lie group magnetic int} 
first through a purely geometric application to groups of Sobolev diffeomorphisms, 
and second through an application to geometric hydrodynamics via the so-called 
\emph{magnetic Euler–Arnold equations}.\medskip

\textbf{Magnetic geodesics on Sobolev diffeomorphism groups.} 
Let \( (M, g) \) be a compact, finite-dimensional Riemannian manifold. 
We consider the group of Sobolev diffeomorphisms \( \mathrm{Diff}^{H^s}(M) \) 
of Sobolev order \( s > \frac{\dim M}{2} + 1 \); for more details, we refer to \Cref{ex: sobolov lie group}. 
We equip this infinite-dimensional half-Lie group with the strong, right-invariant Sobolev metric of order~\( s \), defined by
\begin{equation*}
    \mathcal{G}^s_{\varphi}(h \circ \varphi,\, k \circ \varphi) 
    = \int_M g\big( (1 - \Delta)^{s/2} h,\; (1 - \Delta)^{s/2} k \big)\, \mathrm{dvol}_g, 
    \qquad \forall\, \varphi \in \mathrm{Diff}^{H^s}(M),
\end{equation*}
where \( h \) and \( k \) are \( H^s \)-vector fields on \( M \), and where 
\( \Delta \) denotes the Laplacian with respect to the Riemannian volume form \( \mathrm{dvol}_g \) associated with the Riemannian metric \( g \). 

Using the results of~\cite{EM70,BauerBruverisCismasEscherKolev2020,BauerEscherKolev2015}, 
one sees that \( \mathcal{G}^s \) is indeed a strong Riemannian metric on \( \mathrm{Diff}^{H^s}(M) \). 
We emphasize, however, that for non-integer values of \( s \), this relies on highly non-trivial analytic estimates, 
which have been established only recently~\cite{BauerBruverisCismasEscherKolev2020,BauerEscherKolev2015}. 

As a consequence, we can conclude from~\ref{it: main thm a int}--\ref{it: main thm d int} 
in \Cref{IThm: HopfRinow half lie group magnetic int} that for every right-invariant magnetic system  
\( (\mathrm{Diff}^{H^s}(M), \mathcal{G}^s, \sigma) \), 
with \( \sigma \) as in \Cref{prop: conjugated to Finsler flow}, 
metric completeness holds for all energies above Mañé’s critical value, 
and magnetic geodesic completeness follows. 

To prove magnetic geodesic convexity above Mañé’s critical value, 
one additionally has to verify the extra assumption stated before~\ref{it: main thm e int} 
in \Cref{IThm: HopfRinow half lie group magnetic int}, 
which has been shown implicitly in~\cite{BruverisVialard2017}. 
Thus, for all right-invariant magnetic systems 
\( (\mathrm{Diff}^{H^s}(M), \mathcal{G}^s, \sigma) \) as considered above, 
magnetic geodesic convexity holds for energies above Mañé’s critical value. 

By choosing \( \sigma = 0 \), we recover all completeness results for the group of Sobolev diffeomorphisms 
as obtained in~\cite{BruverisVialard2017}.\medskip

\textbf{Magnetic Euler–Arnold equation.} 
We close this subsection with an application of \Cref{IThm: HopfRinow half lie group magnetic int} 
to geometric hydrodynamics, in the context of the \emph{magnetic Euler–Arnold equation} 
recently introduced by the first author in~\cite{Maier25magneticEulerArnold}. 
By \Cref{IThm: HopfRinow half lie group magnetic int}, we know that for the restriction of the magnetic system 
considered above to the group of Sobolev diffeomorphisms, restricted to its \( C^k \), 
and in particular \( C^1 \) elements, the magnetic geodesic completeness statements 
\ref{it: main thm b int}--\ref{it: main thm d int} hold. 
That is, the magnetic geodesic flow of the system 
\begin{equation}\label{eq: magn sysetm on sobolev lie groups}
    \big( \mathrm{Diff}^{H^{s+1}}(M),\, \mathcal{G}^s,\, \sigma \big)
\end{equation}
exists globally, where we have used that, by \Cref{Ex:Diff_differentiable_elements}, 
the \( C^1 \)-elements in \( \mathrm{Diff}^{H^{s}}(M) \) are precisely \( \mathrm{Diff}^{H^{s+1}}(M) \). 

Using \cite[Lemma 7.4]{Bauer_2025}, the adjoint of \( \mathrm{ad} \) exists, and by the discussion following \eqref{e:mg} 
the Lorentz force of this system exists; thus we can apply \cite[Cor. 2.9, Thm. 2.10]{Maier25magneticEulerArnold}. 
That is, a curve \( \varphi \) is a magnetic geodesic of the system in \eqref{eq: magn sysetm on sobolev lie groups}
if and only if \( u:= \dot{\varphi}\circ\varphi^{-1} \) is a solution of the following partial differential equation:
\begin{equation}\label{eq: MEpDiff}
    m_t + \nabla_u m + (\nabla u)^{\!\top} m + (\operatorname{div} u)\, m = -\, A_s\big(Y_{\id}(u)\big), 
    \qquad m = A_s u, \tag{MEpDiff}
\end{equation}
Here \( A_s = (1 - \Delta_g)^s \) denotes the inertia operator, \( Y \) the Lorentz force of the magnetic system in~\eqref{eq: magn sysetm on sobolev lie groups},
\( m \) is the \emph{momentum density}, 
and \( \nabla \) and \( (\nabla u)^{\!\top} \) are taken with respect to the Levi--Civita connection of \( (M, g) \).\\
We call~\eqref{eq: MEpDiff} the \emph{magnetic EPDiff equation}; in the case of a vanishing magnetic field \( \sigma=0 \) we recover the classical EPDiff equation, the geodesic equation of $(\Diff^{H^s}(M), \G^s)$.

Next, by using the duality between magnetic geodesics of the system~\eqref{eq: magn sysetm on sobolev lie groups} and solutions of \eqref{eq: MEpDiff}, 
together with the fact that the magnetic geodesic flow of \eqref{eq: magn sysetm on sobolev lie groups} exists globally, 
we obtain global well-posedness for \eqref{eq: MEpDiff} for vector fields \( u \) of Sobolev class at least \( s > \frac{\dim M}{2} + 1 \).

\subsection{Outline of the paper.}
In \Cref{s:preliminaries}, we recall the notion of half-Lie groups and their \( C^k \)-elements, 
and we provide background on infinite-dimensional magnetic systems, 
making rigorous in this setting what is well known in finite dimensions. \\
In \Cref{s: Mañé's critical value on the universal cover}, 
we provide additional details concerning the definition of Mañé’s critical value on the universal cover 
and illustrate it by proving \Cref{prop: conjugated to Finsler flow}. 
We also show that there is no vanishing geodesic distance phenomenon 
for the Finsler metric in \Cref{prop: conjugated to Finsler flow}, 
and that the no-loss–no-gain result of~\cite{Bauer_2025} 
extends to magnetic geodesics, 
that is, the magnetic flow does not alter regularity. \\
Finally, in \Cref{s: The Hopf--Rinow theorem for magnetic geodesics on half Lie groups.}, 
we provide a proof of \Cref{IThm: HopfRinow half lie group magnetic int}.\\

\noindent
\textbf{Acknowledgments.} L.M.\ thanks P.~Albers, G.~Benedetti, B.~Khesin, and L.~Assele  for many helpful discussions on Hamiltonian systems. 
L.M.\ would also like to thank M.~Bauer and P.~Michor for insightful discussions about their work~\cite{Bauer_2025}, as well as the participants of \emph{Math en plein air 2025} and \emph{Dynamische Systeme MFO 2025} for valuable feedback on earlier versions of this work.\\
The authors acknowledge funding by the Deutsche Forschungsgemeinschaft (DFG, German Research Foundation) – 281869850 (RTG 2229), 390900948 (EXC-2181/1), and 281071066 (TRR 191). L.M.\ gratefully acknowledges the excellent working conditions and stimulating interactions at the Erwin Schrödinger International Institute for Mathematics and Physics in Vienna during the thematic programme \emph{``Infinite-dimensional Geometry: Theory and Applications''}, where part of this work was carried out. \\
Finally, L.M. thanks F. Schlenk for his warm hospitality at the University of Neuchâtel in October 2025, during which part of this work was completed.

\section{Preliminaries}\label{s:preliminaries}
We begin by fixing some notation. For a group $G$ we denote by $\mu \colon G \times G \to G$ the group multiplication and by $\mu_x, \mu^y$ left and right translations respectively:
\[
\mu(x, y) = \mu_x(y) = \mu^y(x).
\]

\subsection{Half-Lie groups}\label{ss: Half Lie grous}
Let us give a precise definition of half-Lie group and discuss a few examples.
\begin{defn}
    A \textit{right (resp.\ left) half-Lie group} is a topological group endowed with a smooth structure such that right (resp.\ left) multiplication by any element is smooth.
\end{defn}
\begin{rem}
    Half-Lie groups can be modeled on various different types of spaces. In this work we stick to Banach (and Hilbert) half-Lie groups.
\end{rem} 
The main examples include groups of $H^s$- or $C^k$-diffeomorphisms (see~\cite{EM70}), 
as well as semidirect products of a Lie group with kernel an infinite-dimensional representation space. 
Since the first case will serve as our guiding example, we now describe it in more detail.

\begin{example}\label{ex: sobolov lie group}
If $(M, g)$ is a finite-dimensional compact Riemannian manifold or an open Riemannian manifold of bounded geometry, 
then the diffeomorphism group $\mathrm{Diff}^{H^s}(M)$ of Sobolev regularity $s > \dim(M)/2 + 1$ is a half-Lie group. 
Likewise, the groups  $\mathrm{Diff}^{C^k}(M)$ for $1 \leq k < \infty$ 
are half-Lie groups. However, they are not Lie groups because left multiplication is non-smooth.
For a thorough explanation, we refer to~\cite{EM70}.
\end{example}

While left multiplication is only required to be continuous, some elements display better regularity. This is captured by the following definition.
\begin{defn}\label{Def: ck elements}
    An element \( x \in G \) of a Banach half-Lie group is of class \( C^k \) if the left translations \( \mu_x, \mu_x^{-1} \colon G \to G \) are \( k \) times continuously Fréchet differentiable. We denote by \( G^k \) the set of \( C^k \) elements.
\end{defn}

\begin{example}
\label{Ex:Diff_differentiable_elements}
Let $M$ be a closed manifold. Then by~\cite[Ex. 3.5]{Bauer_2025} it holds that
\[
(\mathrm{Diff}^{C^k}(M))^{(\ell)} = \mathrm{Diff}^{C^{k+\ell}}(M).
\]
\end{example}
For further details on $C^k$-elements of half-Lie groups, we refer the reader to \cite[§§2–3]{Bauer_2025}. 

We now recall the notion of regular half-Lie groups in the setting where the half-Lie group carries a Banach manifold structure.
\begin{defn}[Regular half-Lie groups]\label{def:regular half Lie group}
Let $G$ be a Banach half Lie group, and let $F$ be a subset of  $L^1_{\mathrm{loc}}(\RR, T_e G)$.  
We say that $G$ is \emph{$F$–regular} if for every $X \in F$ there exists a unique solution $ g \in W^{1,1}_{\mathrm{loc}}(\RR, G)$ of the differential equation
\[
    \partial_t g(t) \;=\; T_e \mu_{g(t)}\, X(t), 
    \qquad g(0) = e.
\]
This solution is called the \emph{evolution} of $X$ and is denoted by $\mathrm{Evol}(X)$. 
Its evaluation at $t=1$ is denoted by $\mathrm{evol}(X)$.
\end{defn}

\subsection{Riemannian metrics on half-Lie groups}
Following \cite{Bauer_2025}, we recall the notion of strong right-invariant Riemannian metrics.
\begin{defn}\label{Def: strong right invariant Riemannian metrics}
A Riemannian metric \( \mathcal{G} \) on a half-Lie group \( G \) is called 
\emph{$G$-right-invariant} if
\[
  \mathcal{G}_x\big( T_e \mu^x v,\, T_e \mu^x w \big)
  = \mathcal{G}_e(v, w),
  \quad \forall\, x \in G,\, v, w \in T_e G .
\]
The Riemannian metric \( \mathcal{G} \) is called \emph{strong} if 
\( \mathcal{G}_x \) induces the manifold topology on \( T_x G \) for every \( x \in G \); 
otherwise, it is called \emph{weak}.
\end{defn}

Kriegl and Michor \cite{KrieglMichor1997} define \textit{convenient manifolds} as manifolds modeled on convenient vector spaces, i.e. Mackey complete locally convex spaces. All Banach and Fréchet spaces are convenient. Strong Riemannian metrics on such manifolds admit the following useful characterization.
\begin{thm}[{\cite[Theorem 7.2]{Bauer_2025}}]
    Let \( M\) be a convenient manifold equipped with a Riemannian metric $\G$. Then, \( \G \) is a strong Riemannian metric if and only if the canonical map
    \begin{equation*}
\mathcal{G}^\vee \colon TM \longrightarrow T^*M, 
\quad (x, v) \mapsto \mathcal{G}_x(v, \cdot) \,  ,
\end{equation*}
    is a bundle isomorphism.
\end{thm}

\subsection{Infinite dimensional magnetic system}
For a given magnetic system $(M, \mathcal{G}, \sigma)$, consisting of a Hilbert manifold $M$, 
a strong Riemannian metric $\mathcal{G}$, and a closed two-form $\sigma \in \Omega^2(M)$, 
we already mentioned in the introduction that the conservation of energy for magnetic geodesics 
of $(M, g, \sigma)$ reflects the Hamiltonian nature of the system. 
Indeed, recall that the \emph{magnetic geodesic flow} is defined on the tangent bundle by
\[
\Phi_{g,\sigma}^t \colon TM \to TM, \quad 
(q, v) \longmapsto \big( \gamma_{q,v}(t),\, \dot{\gamma}_{q,v}(t) \big), 
\quad \forall\, t \in (a,b),
\]
where $\gamma_{q,v}$ denotes the unique magnetic geodesic with initial condition $(q, v) \in TM$. 
This curve exists since, in this setting, the magnetic geodesic equation~\eqref{e:mg} is locally well-posed by the Picard–Lindelöf theorem. The before mentioned Hamiltonian nature is settled by the following lemma
\begin{lem}\label{lemm: Hamiltonian picture of magnetic geodesic flow}
    The magnetic geodesic flow of $(M,\G, \sigma)$ is the Hamiltonian flow of
    the kinetic energy $E \colon TM \to \RR$ and the twisted symplectic form
\[
\omega_\sigma = \d\lambda - \pi_{TM}^*\sigma,
\]
    where $\lambda$ is the metric pullback of the canonical Liouville $1$-form from $T^*M$ to $TM$ via the metric $g$, and $\pi_{TM} \colon TM \to M$ is the basepoint projection.
\end{lem}
\begin{rem}
    For $\sigma = 0$, we recover in this picture the Hamiltonian formulation of the geodesic flow of $(M,g)$.
\end{rem}
\begin{proof}
First, we observe that the twisted symplectic form $\omega_{\sigma}$ is a strong symplectic form, 
that is, for every $x \in M$, the map 
\[
\omega_{\sigma} : T_x M \longrightarrow T_x^* M
\]
is a topological isomorphism. 

To see this, note first that the canonical symplectic structure on $T^*M$ is strong 
(see~\cite{chernoff1974}), since every Hilbert space is a reflexive Banach space. 
Moreover, since $\mathcal{G}$ is strong, the metric pullback $\mathrm{d}\lambda$ of the canonical symplectic form from $T^*M$ to $TM$ is also strong. 
Next, using the local expression of $\omega_{\sigma}$ around $(x, \alpha) \in T^*M$,
\begin{equation*}
(\omega_{\sigma})_{(x, \alpha)}\big( (e, \beta), (e', \beta') \big) 
= \beta(e') - \beta'(e) - \sigma_x(e, e'),
\end{equation*}
for all $(e, \beta), (e', \beta') \in X \times X^*$, where $X$ denotes the model space of $M$, 
we conclude from the fact that $\mathrm{d}\lambda$ is strong that $\omega_{\sigma}$ is indeed strong symplectic form on $TM$. 
The remaining arguments proceed analogously to the finite-dimensional case of magnetic systems, which can be found for example in \cite{Gin}.
\end{proof}
For a general magnetic system, the zero section corresponds to the set of rest points of the flow. For all energy levels \( \kappa \in (0, \infty) \), all of which are regular values of the kinetic energy \( E \), the corresponding \textit{energy hypersurfaces}
\begin{equation}\label{eq: energy surface}
    \Sigma_{\kappa} := E^{-1}(\kappa),
\end{equation}
are invariant under the magnetic geodesic flow.
Thus, the dynamics of \(\varPhi^t_{g,\sigma}\) restrict to each energy surface \(\Sigma_\kappa\).\\

\noindent\textbf{Exact magnetic systems} are magnetic systems \( (M, g, \mathrm{d}\alpha) \) 
for which the magnetic field is an exact two-form. 
In this case, the magnetic geodesic flow admits a Lagrangian formulation. 
The magnetic Lagrangian associated to \( (M, g, \mathrm{d} \alpha) \) is
\begin{equation}\label{eq: magnetic Lagrangian}
    L \colon TM \to \mathbb{R}, \quad L(q, v) := \tfrac{1}{2} |v|^2 - \alpha_q(v) .
\end{equation}
Next, we aim to show that, in this setting, the magnetic geodesic flow \( \Phi^t_{g,\sigma} \) coincides with the Euler–Lagrange flow \( \Phi_L \) associated with the magnetic Lagrangian \( L \), a fact well known in finite dimensions~\cite{Gin}.
The action of a curve \( \gamma \colon [0, T] \to M \) of class \( H^1 \) is defined by
\begin{equation}\label{eq: def free acttion functional}
    S_L(\gamma) = \int_0^T L \big(\gamma(t), \dot{\gamma}(t) \big) \mathop{}\!\mathrm{d} t.
\end{equation}
Critical points $\gc$ of \( S_L \) satisfy the Euler-Lagrange equations, namely
\begin{equation*}
    \frac{\d}{\d t} (\d_v L)_{(\gamma, \dot{\gamma})} = (\d_x L)_{(\gamma, \dot{\gamma})},
\end{equation*}
where \( \d_x L \) and \( \d_v L \) denote the horizontal and vertical differentials respectively.
\begin{rem}
    Note that the above equation only makes sense in a local coordinate chart. However, since solutions of the Euler-Lagrange equations are precisely the critical points of \( S_L \), the equations do not depend on the choice of local coordinates.
\end{rem}
We begin by proving that minimizers of the action functional \( S_L \) are indeed magnetic geodesics.

\begin{lem}\label{lemma:char_geo}
Let \( (M, \mathcal{G}, \mathrm{d}\alpha) \) be an exact magnetic system, and let 
\( \gamma \colon [a, b] \to M \) be a smooth curve. Then the following are equivalent:
\begin{enumerate}
    \item The curve \( \gamma \) is a magnetic geodesic of \( (M, \mathcal{G}, \mathrm{d}\alpha) \) of energy $\k$; \label{itm:i}
    \item The curve \( \gamma \) is a critical point of the action functional \( S_{L+\k} \), 
    where \( L \) is the magnetic Lagrangian associated with \( (M, \mathcal{G}, \mathrm{d}\alpha) \). \label{itm:ii}
\end{enumerate}
\end{lem}
\begin{proof}
We observe, by a computation following the same lines as in the finite-dimensional case
(see, for instance,~\cite{Abbo13Lect}), that the first variation of \( S_{L+\kappa} \) in \( \gamma \) is given by
\begin{equation*}
\mathrm{d}_{\gamma} S_{L+\kappa}(\eta) 
= \int_{0}^{T} \mathcal{G}_{\gamma(t)} 
\big( \nabla_{\dot{\gamma}} \dot{\gamma} - Y_{\gamma} \dot{\gamma},\, \eta(t) \big) 
\,\mathrm{d}t,
\end{equation*}
for all \( \eta \in \Gamma(\gamma^*TM) \) satisfying \( \eta(0) = \eta(T) = 0 \).
This proves that each critical point of \( S_{L+\kappa} \) corresponds to a solution of the magnetic geodesic equation on \( (M, \mathcal{G}, \mathrm{d}\alpha) \), and conversely, every magnetic geodesic yields a critical point of \( S_{L+\kappa} \).

Moreover, by a standard argument, the variation of \( T \) in the action functional leads to
\[
\frac{\partial S_{L+\kappa}}{\partial T}(\gamma, T)
= \frac{1}{T} \int_0^T \big(\kappa - E(\gamma(t), \dot{\gamma}(t))\big)\, \mathrm{d}t.
\]
Together with the fact that \( E \) is constant along the orbits of the Euler–Lagrange flow,
this identity shows that  \( \gamma \) lies on the energy level
\( E^{-1}(\kappa) \).
This completes the proof.
\end{proof}

To show that the magnetic geodesic flow of \( (M, \mathcal{G}, \sigma) \) coincides with the Euler–Lagrange flow, 
we first need to prove that this flow exists and agrees with its minimizers. 
To describe this, we introduce the following map: 
\begin{equation}\label{eq: defi Legendre transform}
    \mathcal{L} \colon TM \longrightarrow T^*M, 
    \quad (x, v) \mapsto \big( x, (\mathrm{d}_v L)_{(x, v)} \big),
\end{equation}
which is called the \emph{Legendre transform}. 
In the case of a magnetic Lagrangian~\eqref{eq: magnetic Lagrangian}, it takes the form
\begin{equation}\label{eq: legendre transform in case of magnetic Lagrangian}
    \mathcal{L}(x, v) = \big( x,\, \G_x( v, \cdot ) - \alpha_x(\cdot) \big)\quad \forall(x,v)\in TM.
\end{equation}
Since \( \mathcal{G} \) is strong by assumption, the map \( \mathcal{L} \) is a bundle isomorphism. 
Hence, the magnetic Lagrangian \( L \) is \emph{hyperregular} in the sense of~\cite[§7.3]{MardsenRatiuMechanics}. 
By~\cite[Thm.~7.3.3]{MardsenRatiuMechanics}, the flow lines of the Euler–Lagrange flow \( \Phi_L \) of \( L \) 
are precisely the solutions of the Euler–Lagrange equations associated with \( L \). 
Thus, using \Cref{lemma:char_geo}, we obtain the following.

\begin{lem}\label{lem: magnetic flow coincides with ELflow}
The magnetic geodesic flow \( \Phi_t^{\mathcal{G},\sigma} \) of the exact magnetic system \( (M, \mathcal{G}, \mathrm{d}\alpha) \) 
coincides with the Euler–Lagrange flow \( \Phi_L \) of the magnetic Lagrangian \( L \) associated with \( (M, \mathcal{G}, \mathrm{d}\alpha) \) defined as in~\eqref{eq: magnetic Lagrangian}.
\end{lem}
\begin{rem}
In forthcoming work~\cite{TonelliHalfLiegroups}, 
the authors show that the constrcution underlying \Cref{lem: magnetic flow coincides with ELflow} fits into a more general framework, 
namely that of so-called Tonelli Lagrangians.
\end{rem}
For the sake of completeness, we also note that it admits a natural dual counterpart. 
We define the Hamiltonian 
\begin{equation}\label{eq: magnetic Hamiltonian}
    H := E \circ \mathcal{L}^{-1} \colon T^*M \longrightarrow \mathbb{R},
\end{equation}
as the Legendre dual of the magnetic Lagrangian \( L \) defined in~\eqref{eq: magnetic Lagrangian}. 
Using this explicit description together with~\eqref{eq: legendre transform in case of magnetic Lagrangian}, we obtain
\[
H(x, p) = \frac{1}{2} | p + \alpha_x |_{\mathcal{G}}^2, 
\qquad \forall\, (x,p) \in T^*M.
\]

Since the canonical symplectic form \( \omega_{\mathrm{can}} \) on \( T^*M \) 
is a strong symplectic form, the Hamiltonian \( H \) induces a unique vector field 
\( X_H \in \mathfrak{X}(T^*M) \) defined by
\begin{equation*}
    \iota_{X_H} \omega_{\sigma} = \mathrm{d}H,
\end{equation*}
called the \emph{Hamiltonian vector field}. Since \( TM \) is a Hilbert manifold and \( X_H \) is a smooth vector field on \( TM \), 
the equations of motion
\[
\dot{\Gamma} = X_H(\Gamma)
\]
are locally well-posed by the Picard–Lindelöf theorem. 
This defines a local flow \( \Phi_t^H \) on \( T^*M \) for \( t \) small enough. 

As \( H \) is, by its definition~\eqref{eq: magnetic Hamiltonian}, 
the Legendre transform of the magnetic Lagrangian \( L \), which is hyperregular, 
it is hyperregular in the sense of~\cite[§7.4]{MardsenRatiuMechanics}. 
Hence, by~\cite[Thm.~7.4.3]{MardsenRatiuMechanics}, 
the Hamiltonian flow of \( H \) is conjugated via the Legendre transform 
to the Euler–Lagrange flow of \( L \). 
Therefore, by \Cref{lem: magnetic flow coincides with ELflow}, we obtain the following.

\begin{cor}
Let \( (M, \mathcal{G}, \mathrm{d}\alpha) \) be an exact magnetic system. 
Then its magnetic geodesic flow \( \Phi_t^{\mathcal{G}, \mathrm{d}\alpha} \) 
is conjugated, via the Legendre transform~\eqref{eq: defi Legendre transform} 
induced by the magnetic Lagrangian \( L \), 
to the Hamiltonian flow of the magnetic Hamiltonian 
\( H \) given in~\eqref{eq: magnetic Hamiltonian} on \( T^*M \).
\end{cor}

\section{Mañé's critical value on the universal cover}\label{s: Mañé's critical value on the universal cover}
This section is devoted to defining Mañé's critical value on the universal cover 
for right-invariant magnetic systems on half-Lie groups. 
We then illustrate this energy threshold by showing that, for energies above it, 
the lift of the magnetic geodesic flow to the universal cover 
coincides with a Finsler geodesic flow.

\subsection{Definition of the Mañé critical value on the universal cover}
To define Mañé's critical value \( c(G, \mathcal{G}, \sigma) \) for the universal cover \( \hat{G} \) of \( G \), we first observe the following. \\
The universal cover of \( (G, \mathcal{G}) \) exists because, by~\cite[Thm.~7.2]{Bauer_2025}, 
the half-Lie group \( G \) equipped with the strong Riemannian metric \( \mathcal{G} \) 
is a Hilbert manifold endowed with a strong Riemannian metric. 
Its universal cover \( \pi \colon \hat{G} \to G \) is again a topological group, and by lifting the smooth structure from \( G \) 
to the universal cover \( \hat{G} \), the latter also admits a half-Lie group structure such that 
\( \pi \) becomes a smooth topological group homomorphism. \\
Moreover, if \( \mathcal{G} \) is a strong Riemannian metric on \( G \), 
then the lifted metric \( \hat{\mathcal{G}} := \pi^*\mathcal{G} \) 
is a strong Riemannian metric on \( \hat{G} \), 
since strongness is a purely local property of a Riemannian metric. 
Finally, right-invariance of \( \mathcal{G} \) implies right-invariance of \( \hat{\mathcal{G}} \), 
because right multiplication \( \hat{\mu}^x \) on \( \hat{G} \) satisfies 
\[
\pi \circ \hat{\mu}^x = \mu^{\pi(x)} \circ \pi
\qquad \text{for all } x \in \hat{G}.
\]
Recall that a closed two-form \( \sigma \) on \( G \) is called \emph{weakly exact} 
if \( \sigma\!\mid_{\pi_2(G)} = 0 \), that is, if the lift \( \hat{\sigma} \) of \( \sigma \) 
to the universal cover \( \hat{G} \) is exact. 
We denote by \( \hat{\mathcal{G}} \) the lift of \( \mathcal{G} \) to \( \hat{G} \), 
which is \( \hat{G} \)-right invariant and again a strong Riemannian metric, 
as shown above.

With this notation in place, we can now introduce the notion of
\emph{Mañé's critical value} for magnetic systems on half-Lie groups.
\begin{defn}[=\Cref{Def: int maneuni half int}]\label{Def: int maneuni half}
Let \((G, \mathcal{G}, \sigma)\)  be a \( G \)-right-invariant magnetic system.  If $\sigma$ is weakly exact we define the \emph{Mañé critical value} of the universal cover of the magnetic system \((G, \mathcal{G}, \sigma)\) by
\begin{equation}\label{e:maneuni_half}
c(G, \mathcal{G}, \sigma) := \inf_{\mathrm{d}\hat{\alpha} = \hat{\sigma}} \frac{1}{2} \,\Vert \hat{\alpha} \Vert_\infty^2 := \inf_{\mathrm{d}\hat{\alpha} = \hat{\sigma}} \frac{1}{2} \sup_{\hat x\in \hat{G}}\,\vert \hat{\alpha}_{\hat x} \vert_{\hat{\G}}^2
\;\in [0,\infty],
\end{equation}
where the infimum is taken over all $\hat{G}$-right-invariant primitives $\hat{\alpha}$ of $\hat{\sigma}$ (and it is understood to be \( \infty \) should such a primitive not exist). If \( \sigma \) is not weakly exact, we set \( c(G, g, \sigma) = \infty \).
\end{defn}

\subsection{The magnetic geodesic flow as a Finsler flow}
This section is devoted to proving \Cref{prop: conjugated to Finsler flow}. We recall for the convenience of the reader here
\begin{prop}[=\Cref{prop: conjugated to Finsler flow}]\label{prop: conjugated to Finsler flow_sec}
Let $(G, \mathcal{G}, \sigma)$ be a $G$-right-invariant magnetic system with $\sigma$ weakly exact and assume its pullback $\hat{\sigma}$ to the universal cover $\hat{G}$ admits a $\hat{G}$-right-invariant primitive \( \hat{\alpha} \).
    
    Then, for every \( \k > c(G, \mathcal{G}, \sigma) \), we have
    \begin{equation}\label{e:finsler_metric_sec}
        \mathcal{F}^{\kappa} \colon T\hat{G} \to \mathbb{R}, 
        \qquad 
        \mathcal{F}^{\kappa}(\hat{x}, \hat{v}) := \sqrt{2k}\, |\hat{v}|_{\hat{x}} - \hat{\alpha}_{\hat{x}}(\hat{v}),
        \quad 
        \forall\, (\hat{x}, \hat{v}) \in T\hat{G},
    \end{equation}
    defines a $\hat{G}$-right-invariant Finsler metric on \( \hat{G} \) satisfying the following properties:
    \begin{enumerate}
        \item \label{it: 1 finsler flow conj} There exist constants \( C_1, C_2 > 0 \) such that
        \[
            C_1\, \mathcal{F}^\kappa_{e}(\hat{v})
            \;\le\; \sqrt{\hat{\mathcal{G}}_{e}(\hat{v},\hat{v})}
            \;\le\; C_2\, \mathcal{F}^\kappa_{e}(v),
            \qquad 
            \forall\, \hat{v} \in T_e\hat{G} \setminus \{0\}.
        \]
       \item \label{it: 2 finsler flow conj} On the energy surface \( \hat{E}^{-1}(\kappa) := \Sigma_{\kappa} \subseteq T\hat{G} \), as defined in~\eqref{eq: energy surface}, 
the lift of the magnetic flow of \( (G, \mathcal{G}, \sigma) \) 
to the universal cover \( \hat{G} \) coincides with the Finsler geodesic flow of \( (\hat{G}, \mathcal{F}^{\kappa}) \).
    \end{enumerate}
\end{prop}
Before doing so, we recall some useful facts about infinite-dimensional Finsler geometry, 
following the exposition in~\cite[Appendix]{bauerLpFisherRaoMetricAmari2024}.\medskip

\noindent\textbf{Infinite-dimensional Finsler geometry.}
Let \( \mathcal{M} \) be an infinite-dimensional Fréchet manifold with tangent bundle \( T\mathcal{M} \), 
equipped with a strongly convex Finsler structure 
\( \mathcal{F} \colon T\mathcal{M} \to \mathbb{R} \) in the sense of~\cite[Def.~A.1]{bauerLpFisherRaoMetricAmari2024}. \\[0.3em]
For a piecewise smooth path \( \gamma \colon [a, b] \to \mathcal{M} \), 
the length of \( \gamma \) with respect to \( \mathcal{F} \) is defined by
\begin{equation}\label{eq: Finsler length}
    \mathbb{L}_{\mathcal{F}}(\gamma) 
    = \int_a^b \mathcal{F}_{\gamma(t)}\big( \dot{\gamma}(t) \big)\, \mathrm{d}t.
\end{equation}
For any pair of points \( x, y \in \mathcal{M} \), we consider the induced geodesic distance function
\begin{equation}\label{eq: Finsler geodesic distance}
    d_{\mathcal{F}}(x, y) := \inf_{\gamma} \mathbb{L}_{\mathcal{F}}(\gamma),
\end{equation}
where the infimum is taken over all piecewise smooth curves \( \gamma \) connecting \( x \) to \( y \). \\[0.3em]
As in Riemannian geometry, one can show that minimizing the length is equivalent to minimizing 
the \emph{energy functional}, defined by
\begin{equation}\label{eq: def Finsler energy}
    \mathbb{E}_{\mathcal{F}}(\gamma) := \int_a^b \mathcal{F}^2_{\gamma(t)}\big( \dot{\gamma}(t)\big)\, \mathrm{d}t.
\end{equation}
In Riemannian geometry, minimizers of the energy functional are precisely the solutions of the geodesic equation with respect to the Levi--Civita connection. 
In Finsler geometry, this role is played by the so-called Chern connection, as the following lemma shows. 
For a precise definition, we refer to~\cite[Def.~A.6]{bauerLpFisherRaoMetricAmari2024}.
\begin{lem}[{\cite[Lemma~A.9]{bauerLpFisherRaoMetricAmari2024}}]\label{lem:Finsler geodesics chern connectiosn}
Let \( (\mathcal{M}, \mathcal{F}) \) be as above.
Assume in addition that the Chern connection \( \nabla^{\mathcal{F}} \) exists. 
Then the critical points of the energy functional~\eqref{eq: def Finsler energy} 
are precisely the curves satisfying the geodesic equation
\begin{equation}\label{eq: finsler geodesic eq}
    \nabla^{\mathcal{F}}_{\dot{\gamma}(t)} \dot{\gamma}(t) = 0.
\end{equation}
\end{lem}

\medskip
The rest of this subsection is devoted to proof \Cref{prop: conjugated to Finsler flow_sec}.
\begin{proof}[\textbf{Proof of \Cref{prop: conjugated to Finsler flow_sec}}]  We start with~\ref{it: 1 finsler flow conj}. 
First, we observe that \( \mathcal{F}^{\kappa} \), as defined in~\eqref{e:finsler_metric_sec}, 
is, outside of the zero section, a \( C^{\infty} \)-map, since it is the sum of two \( C^{\infty} \)-maps on \( T\hat{G} \setminus \{0\} \). 
By its definition~\eqref{e:finsler_metric_sec}, it follows that \( \mathcal{F}^{\kappa} \) 
scales linearly in the fibres and is subadditive in the fibres.

Next, we have to check that \( \mathcal{F}^{\kappa} \) is positive on \( T\hat{G} \setminus \{0\} \). 
Since \( \hat{\G} \) and \( \hat{\alpha} \), and therefore \( \Vert \cdot \Vert_{\hat{\G}} \), 
are \( \hat{G} \)-right-invariant, it suffices to verify this at \( \hat{e} \in \hat{G} \). 
Using that \( \hat{\alpha}_{\hat{e}} \) is a bounded linear functional, we obtain from~\eqref{e:finsler_metric_sec}
\begin{equation}\label{eq: finsler metric lower bound}
    \mathcal{F}^{\kappa}_{\hat{e}}(\hat{v}) 
    = \sqrt{2 \kappa}\, | \hat{v} | - \hat{\alpha}_{\hat{e}}(\hat{v})
    \;\geq\;
    \sqrt{2 \kappa}\, | \hat{v} | 
    - \Vert \hat{\alpha} \Vert_{\infty}\, | \hat{v} |,
    \qquad \forall\, \hat{v} \in T_{\hat{e}}\hat{G}\,,
\end{equation}
where \( \Vert \hat{\alpha} \Vert_{\infty} \) denotes the operator norm of \( \hat{\alpha}_{\hat{e}} \) induced by \( \hat{\G} \). 
Setting \( C_1 = \sqrt{2 \kappa} - \Vert \hat{\alpha} \Vert_{\infty} \), we obtain from~\eqref{eq: finsler metric lower bound}
\begin{equation}\label{eq: lower bound Finsler norm}
    C_1\, | \hat{v} | 
    \leq \mathcal{F}^{\kappa}_{\hat{e}}(\hat{v}), 
    \qquad \forall\, \hat{v} \in T_{\hat{e}}\hat{G}\,.
\end{equation}
Hence \( \mathcal{F}^{\kappa} \) is positive on \( T\hat{G} \setminus \{0\} \), and 
\( \mathcal{F}^{\kappa}_{\hat{e}}(\hat{v}) = 0 \) if and only if \( \hat{v} = 0 \). 
Thus, \( \mathcal{F}^{\kappa} \) defines a Finsler metric on \( \hat{G} \). 

Using again that \( \hat{\alpha}_{\hat{e}} \) is a bounded linear operator and 
that \( \mathcal{F}^{\kappa} \) is \( \hat{G} \)-right-invariant, 
we obtain the upper bound by setting \( C_2 = \sqrt{2 \kappa} + \Vert \hat{\alpha} \Vert_{\infty} \):
\begin{equation}\label{eq: finsler metric upper bound}
    \mathcal{F}^{\kappa}_{\hat{e}}(\hat{v})
    \leq C_2\, | \hat{v} |, 
    \qquad \forall\, \hat{v} \in T_{\hat{e}}\hat{G}\,.
\end{equation}
From~\eqref{eq: lower bound Finsler norm} and~\eqref{eq: finsler metric upper bound}, 
we can conclude that \( \mathcal{F}^{\kappa} \) satisfies the inequality in~\ref{it: 1 finsler flow conj} 
of \Cref{prop: conjugated to Finsler flow_sec}. \\
Finally, we point out that in finite dimensions it is well known that Randers metrics are always strongly convex. 
By a computation analogous to the finite-dimensional case, 
one shows that \( \mathcal{F}^{\kappa} \) is also strongly convex, 
since it has precisely the form of a Randers metric.\medskip

Part~\ref{it: 2 finsler flow conj}. 
We begin by observing that, by assumption, the lift of the magnetic system \( (G, \G, \sigma) \) 
to its universal cover is an exact magnetic system \( (\hat{G}, \hat{\G}, \mathrm{d}\hat{\alpha}) \). 
Thus, by \Cref{lem: magnetic flow coincides with ELflow}, 
the magnetic geodesic flow of \( (\hat{G}, \hat{\G}, \mathrm{d}\hat{\alpha}) \) 
coincides with the Euler--Lagrange flow of the magnetic Lagrangian \( L \) 
defined in~\eqref{eq: magnetic Lagrangian}. 
Hence, by \Cref{lemma:char_geo}, 
a curve \( \gamma \) is a magnetic geodesic of \( (\hat{G}, \hat{\G}, \mathrm{d}\hat{\alpha}) \) 
with energy \( \kappa \) if and only if it is a critical point of the action functional 
\( S_{L + \kappa} \) as defined in~\eqref{eq: def free acttion functional}.
Therefore, it suffices to show that the critical points of \( S_{L + \kappa} \)
are precisely the length minimizers of \( \mathbb{L}_{\mathcal{F}^{\kappa}} \), 
as defined in~\eqref{eq: Finsler length}, and vice versa. 

 For this, we first fix some notation: for given \( \hat{x}, \hat{y} \in \hat{G} \), we denote by  
\begin{equation*}
    \Omega = \{ \gamma \in H^1([0,T], \hat{G}) \mid T > 0,\ \gamma(0) = \hat{x},\ \gamma(T) = \hat{y},\ |\dot{\gamma}| \equiv \sqrt{2\kappa} \}.
\end{equation*}
the space of all \( H^1 \)-curves of constant speed \( \sqrt{2\kappa} \) connecting \( \hat{x} \) and \( \hat{y} \).  
Note that the condition \( |\dot{\gamma}| \equiv \sqrt{2\kappa} \) means that all curves in \( \Omega \) lie on the energy surface \( \Sigma_\k \).  

We are going to show that the minimization problems  
\begin{equation*}
    \inf_{\gamma \in \Omega} \mathbb{L}_{\mathcal{F}^\kappa}(\gamma), 
    \qquad 
    \inf_{\gamma \in \Omega} S_{L + \kappa}(\gamma)
\end{equation*}
are equivalent.  

Note that \( \mathbb{L}_{\mathcal{F}^\kappa} \) is 1-homogeneous in \( \dot{\gamma} \), and hence invariant under time reparametrization, while \( S_{L + \kappa} \) is not.  
Since \( a^2 + b^2 \ge 2ab \), we have  
\begin{equation*}
    \sqrt{2\kappa}\,|\dot{\gamma}| \le \tfrac{1}{2}|\dot{\gamma}|^2 + \kappa,
\end{equation*}
with equality if and only if \( \|\dot{\gamma}\| \equiv \sqrt{2\kappa} \).  
This implies that \( \mathbb{L}_{\mathcal{F}^\kappa}(\gamma) \le S_{L + \kappa}(\gamma) \) for any curve \( \gamma \), and equality holds precisely when the curve is parameterized with constant speed \( \sqrt{2\kappa} \).  
Thus, the statement follows from the fact that \( \mathbb{L}_{\mathcal{F}^\kappa} \) and \( S_{L + \kappa} \) coincide on \( \Omega \).\medskip

In summary we have proven that lifts of magnetic geodesics of the system \( (G,\G,\sigma) \) of energy \( \kappa \) to the universal cover \( \hat G \) are precisely minimizers of the Finsler length \( \mathbb{L}_{\mathcal{F}^\kappa} \). 
What remains is to prove that the minimizers of the Finsler length \( \mathbb{L}_{\mathcal{F}^\kappa} \) are indeed flow lines of the Finsler geodesic flow of \( (\hat{G}, \mathcal{F}^{\kappa}) \); for this it suffices, in light of \Cref{lem:Finsler geodesics chern connectiosn}, to prove that the Chern connection of \( (\hat{G},\mathcal{F}_{\kappa}) \) exists. 
For what follows, we use the explicit form of \( \mathcal{F}_{\kappa} \) to prove that in this specific situation the argument from the finite-dimensional case goes through, included for the sake of completeness in Appendix \ref{appendix:chern_connection}.    
\end{proof}

\subsection{Nondegeneracy of the geodesic distance and no-gain no-loss}
In this subsection, we continue the study of the right-invariant magnetic system
\( (G, \mathcal{G}, \sigma) \) introduced in \Cref{prop: conjugated to Finsler flow_sec}.
From that result, we have seen that this magnetic system gives rise to a Finsler metric
\( \mathcal{F}^\kappa \) on the universal cover \( \hat{G} \) for all \( \kappa > c(G, \mathcal{G}, \sigma) \).
Following~\eqref{eq: Finsler geodesic distance}, we denote by
\( d_{\mathcal{F}^\kappa} \) the geodesic distance associated with the Finsler manifold
\( (\hat{G}, \mathcal{F}^\kappa) \).
With this notation in place, we can now proceed to prove the following result.
\begin{lem}
    The geodesic distance $d_{{\mathcal{F}}^\kappa}$ on $\hat{G}$ is, for all $\kappa > c(G, \mathcal{G}, \sigma) $, non-degenerate; that is, the normal subgroup of elements with vanishing geodesic distance is trivial.
\end{lem}
\begin{proof}
    By \cite[Thm. 7.3]{Bauer_2025} the geodesic distance of $\hat{G},\hat \G$ is non-degenerate, which in combination with \ref{it: 1 finsler flow conj} in \Cref{prop: conjugated to Finsler flow_sec} finihes the proof.
\end{proof}

We close this subsection by presenting a result of independent interest. 
We begin by recalling a recent theorem of Bauer–Harms–Michor~\cite[Section~7]{Bauer_2025}, 
where the authors show that the exponential map of a strong right-invariant metric on a 
half-Lie group \( G \) restricts to the subgroups \( G^\ell \) of \( C^\ell \)-elements in \( G \).
In particular, this means that there is neither gain nor loss of regularity along geodesics.  
We now aim to establish an analogous statement for the magnetic geodesic flow.

\begin{prop}[No-loss-no-gain]\label{prop: no gain no loose}
Let \( (G, \mathcal{G}, \sigma) \) be a right-invariant magnetic system. 
Then, for every \( \ell \geq 1 \), the magnetic geodesic flow defines a smooth map
\[
\Phi_t^{\mathcal{G}, \sigma} : TG^\ell \longrightarrow TG^\ell,
\]
for all \( t \) sufficiently small. In particular, the evolution along magnetic geodesics preserves regularity: 
there is neither gain nor loss of smoothness.
\end{prop}

\begin{proof}
By \Cref{lemm: Hamiltonian picture of magnetic geodesic flow}, 
the magnetic geodesic flow of \( (G, \mathcal{G}, \sigma) \) is precisely the Hamiltonian flow 
of the kinetic energy \( E(x,v) = \tfrac{1}{2}\,\mathcal{G}_x(v,v) \) 
with respect to the twisted symplectic structure \( \omega_\sigma \). 
Since \( \omega_\sigma \) is a strong symplectic form, this flow is locally well-posed by the Picard–Lindelöf theorem. 
As \( \sigma \) and \( \mathcal{G} \) are \( G \)-right-invariant, the same holds for \( E \) and \( \omega_\sigma \). 
Hence, this Hamiltonian flow satisfies the assumptions of~\cite[Thm.~7.5]{Bauer_2025}, 
from which the claim follows.
\end{proof}

As in the case of the geodesic flow, which gives rise to the exponential map, 
the magnetic geodesic flow of \( (G, \mathcal{G}, \sigma) \) gives rise to the 
\emph{magnetic exponential map} of the magnetic system \( (G, \mathcal{G}, \sigma) \), which is defined by
\begin{equation}\label{Magnetic exponential map}
    \exp^{\mathcal{G}, \sigma} \colon \mathcal{U} \subseteq TG \longrightarrow G, \quad
    (x,v) \longmapsto \gamma_{x,v}(1),
\end{equation}
where \( \gamma_{x,v} \) denotes the unique magnetic geodesic of \( (G, \mathcal{G}, \sigma) \) 
with initial condition \( (x,v) \).  Here, \( \mathcal{U} \subseteq TG \) is the maximal open set where \( \exp^{\mathcal{G}, \sigma} \) is defined.

By \Cref{prop: no gain no loose}, it follows directly th
\begin{cor}[No-loss–no-gain for \( \exp^{\mathcal{G}, \sigma} \)]
\label{cor: no-loss-no-gain for exp}
Let \( (G, \mathcal{G}, \sigma) \) be as in \Cref{prop: no gain no loose}, and let 
\( \mathcal{U} \subset TG \) be the maximal open neighborhood of the zero section 
on which \( \exp^{\mathcal{G}, \sigma} \) is smooth.  
Then, for every \( \ell \geq 1 \), the magnetic exponential map restricts to a smooth map
\[
\exp^{\mathcal{G}, \sigma} \colon \mathcal{U} \cap TG^\ell \longrightarrow G^\ell,
\]
that is, evolution along magnetic geodesics preserves regularity: 
there is no gain or loss of smoothness.%
\end{cor}
\section{The Hopf--Rinow theorem for magnetic geodesics on half Lie groups.}  \label{s: The Hopf--Rinow theorem for magnetic geodesics on half Lie groups.}
This section is devoted to proving the main theoretical result of this article, namely \Cref{IThm: HopfRinow half lie group magnetic int}, which we recall here for the convenience of the reader.
\begin{thm}\label{IThm: HopfRinow half lie group magnetic}
  Let $(G,\G,\s)$ be as in \Cref{prop: conjugated to Finsler flow_sec}. Then the following holds true
\begin{enumerate}
    \item \label{it: main thm a} For all energy levels \( \kappa > c(G,\G,\sigma) \), with \( c(G,\G,\sigma )\) as in~\eqref{e:maneuni_half}, the space \( (\hat G,  d_{\mathcal{F}^\kappa}) \) is a complete metric space, i.e., every \(  d_{\mathcal{F}^\kappa}\)-Cauchy sequence converges in \( \hat G \).
    
    \item \label{it: main thm b} The magnetic exponential map \( \exp^{\G,\sigma}_e : T_e G \to G \) is defined on all of \( T_e G \).
    
    \item \label{it: main thm c} The magnetic exponential map \( \exp^{\G,\sigma} : TG \to G \) is defined on all of \( TG \).
    
    \item \label{it: main thm d} The magnetic system \( (G, \G, \sigma) \) is magnetically geodesically complete, i.e., every magnetic geodesic is maximally defined on all of \( \mathbb{R} \).
\end{enumerate}
Assume in addition that \( \hat{G} \) is \( L^2 \)-regular and that for each \( x \in \hat G \), the sets
\[
A_{x} := \left\{ \xi \in L^2([0,1], T_e \hat G) : \mathrm{evol}(\xi) = x \right\} \subset L^2([0,1], T_e \hat G)
\]
are weakly closed. Then:
\begin{enumerate}
    \setcounter{enumi}{4} 
    \item \label{it: main thm e} For all energy levels \( \kappa > c(G, \G, \sigma) \) and all \( x, y \in G \), there exists a magnetic geodesic of \( (G, \G, \sigma) \) of energy \( \kappa \) 
that minimizes the action connecting \( x \) and \( y \).
\end{enumerate}
In addition, the geodesic completeness statements, items~\ref{it: main thm b}- \ref{it: main thm d}, also hold for the magnetic systems \( (G^\ell, \G, \sigma) \) for all $\ell\geq 1$ on the weak Riemannian manifolds \( (G^\ell, \G) \), where \( \G \), respectively \( \sigma \), denotes the restriction of the Riemannian metric and the magnetic two-form from \( G \) to \( G^\ell \). 

\end{thm}

\begin{proof}
 We begin with \ref{it: main thm a}. 
By \ref{it: 1 finsler flow conj} in \Cref{prop: conjugated to Finsler flow_sec}, 
the norm induced by \( \hat{\mathcal{G}} \) on \( \hat{G} \) is equivalent to the Finsler metric \( \mathcal{F}^\kappa \). 
Using the inequality established there, together with the fact that 
\( (\hat{G}, d_{\hat{\mathcal{G}}}) \) is, by \cite[Thm.~7.7~(a)]{Bauer_2025}, 
a complete metric space, concludes the proof. \medskip

\noindent For items~\ref{it: main thm b}–\ref{it: main thm d}, we adapt the second proof strategy in
\cite[Thm.~7.7]{Bauer_2025} for the corresponding statements (b)–(d) therein.
More precisely, we follow an argument similar to \cite[Lemma~6.2]{GayBalmazRatiu2015},
using the description of the magnetic geodesic flow as a Hamiltonian flow on the tangent
bundle (see \Cref{lemm: Hamiltonian picture of magnetic geodesic flow}). \medskip

\noindent We close by giving a proof of item~\ref{it: main thm e}. 
By \ref{it: 2 finsler flow conj} in \Cref{prop: conjugated to Finsler flow_sec}, 
the lift of the magnetic geodesic flow of \( (G, \mathcal{G}, \sigma) \) to the universal cover 
\( \hat{G} \) coincides on the energy surface \( \Sigma_\kappa \) with the Finsler geodesic flow 
of \( (\hat{G}, \mathcal{F}^\kappa) \) for all \( \kappa > c(G, \mathcal{G}, \sigma) \). 
Thus, it suffices to show that for all pairs of points 
\( \hat{x}, \hat{y} \in \hat{G} \) and all \( \kappa > c(G, \mathcal{G}, \sigma) \), 
there exists a Finsler geodesic of \( (\hat{G}, \hat{\mathcal{F}}^\kappa) \) connecting 
\( \hat{x} \) and \( \hat{y} \) in \( \hat{G} \), where 
\( \hat{\mathcal{F}}^\kappa \) is defined as in~\eqref{e:finsler_metric_sec}. 

By the discussion following~\eqref{eq: def Finsler energy}, 
the geodesic distance of \( (\hat{G}, \hat{\mathcal{F}}^\kappa) \) between 
\( \hat{x} \) and \( \hat{y} \) can be computed by minimizing the Finsler energy functional of \( \mathcal{F}^\kappa \)
\begin{equation}\label{eq: finsler energy}
    \mathbb{E}(\gamma) 
    = \int_0^1 
        \big( \hat{\mathcal{F}}^\kappa_{\gamma(t)}(\dot{\gamma}(t)) \big)^2 
        \, \mathrm{d}t,
\end{equation}
over all paths \( \gamma \in H^1([0,1], \hat{G}) \) such that 
\( \gamma(0) = \hat{x} \) and \( \gamma(1) = \hat{y} \). 

Now, by the \( L^2 \)-regularity of \( \hat{G} \) and the right-invariance of \( \mathcal{F}^\kappa \), 
the problem of finding a minimizer of~\eqref{eq: finsler energy} 
is equivalent to finding a minimizer of the functional
\begin{equation}\label{eq: finsler functional restricted to subset A}
    A_{\hat{y}\hat{x}^{-1}} \ni \xi \longmapsto 
    \int_0^1 
        \mathcal{F}^\kappa(e, \xi(t))^2 
        \, \mathrm{d}t 
    \; =: \; \mathbb{E}(\xi),
\end{equation} 
where \( A_{\hat{y}\hat{x}^{-1}} := \{\, \xi \in L^2([0,1], T_e\hat{G}) \mid \mathrm{evol}(\xi) = \hat{y}\hat{x}^{-1} \,\} \).

Note that, by \ref{it: 1 finsler flow conj} in 
\Cref{prop: conjugated to Finsler flow_sec}, 
\( \mathcal{F}^\kappa \) and \( \Vert \cdot \Vert \) define the same topology on \( T_e G \); 
the only difference is that \( (T_e G, \mathcal{F}^\kappa) \) is no longer a Hilbert space, 
but merely a Banach space. 
Hence, the (weak) \( L^2 \)-topology on \( L^2([0,1], T_{\hat{e}} \hat{G}) \) 
can be equivalently defined using \( \mathcal{F}^\kappa \).

The functional \( \mathbb{E} \) is sequentially weakly lower semicontinuous, 
since for any Banach space \( (X, \Vert \cdot \Vert_X) \) and any \( p > 0 \), 
the map \( x \mapsto \Vert x \Vert_X^p \) is weakly lower semicontinuous. 
Moreover, \( \mathcal{E} \) is coercive, i.e.
\begin{equation*}
    \int_0^1 \mathcal{F}^\kappa(e, \xi(t))^2 \, \mathrm{d}t \;\to\; \infty
    \quad \text{as } \Vert \xi \Vert_{L^2} \to \infty.
\end{equation*}
This immediately implies that the infimum of \( \mathbb{E} \) 
as in~\eqref{eq: finsler functional restricted to subset A} 
over \( A_{\hat{y}\hat{x}^{-1}} \) is finite. \\
Let \( (\xi_n) \subset A_{\hat{y}\hat{x}^{-1}} \) be a minimizing sequence, i.e\ 
\[
\left\vert\inf_{\xi \in A_{\hat{y}\hat{x}^{-1}}} 
    \Vert \xi \Vert_{L^2([0,1], T_e \hat{G})}^2
    - \Vert \xi_n \Vert_{L^2([0,1], T_e \hat{G})}^2
    \right\vert< \frac{1}{n}
    \quad \forall\, n \in \mathbb{N}.
\]
By coercivity of the Finsler energy Functional $\mathbb{E}$, the sequence \( (\xi_n) \) is uniformly bounded in \( L^2([0,1], T_e \hat{G}) \). 
Using the assumption that the set 
\[
A_x := \{\, \xi \in L^2([0,1], T_e G) \mid \mathrm{evol}(\xi) = x \,\}
\]
is weakly closed for all \( x \in G \), 
it follows that \( A_{\hat{x}} \) is weakly closed for all \( \hat{x} \in \hat{G} \). 
Hence, by the Eberlein--\v{S}mulyan theorem, 
we may assume (after passing to a subsequence) that \( \xi_n \rightharpoonup \xi \) weakly in \( L^2 \), 
for some \( \xi \in A_{\hat{y}\hat{x}^{-1}} \). 
Finally, by sequential weak lower semicontinuity of \( \mathbb{E} \), we obtain
\[
\mathbb{E}(\xi) 
\le \liminf_{n \to \infty} \mathbb{E}(\xi_n)
= \inf_{\xi \in A_{\hat{y}\hat{x}^{-1}}} \mathbb{E}(\xi),
\]
which shows that \( \xi \) is a minimizer.

Hence, \( \gc := \mathrm{Evol}(\xi) \) is an Finsler energy-minimizing path connecting \( \hat x \) and \( \hat y\).  By the proof of \ref{it: 2 finsler flow conj} in \Cref{prop: conjugated to Finsler flow_sec} it is an magnetic geodesic of $(\hat{G},\hat{\G}, \hat\s)$ of energy $\k$ connecting $\hat{x}$ and $\hat{y}$, here the exact magnetic system $(\hat{G},\hat{\G}, \hat{\s})$ is the lift of the magnetic system $(G,\G,\s)$ to the universal cover $\pi:\hat G\to G$.

Let \( \hat{x}, \hat{y} \) be points in the fibers 
\( \pi^{-1}(x) \) and \( \pi^{-1}(y) \), respectively. 
By the above argument, \( \hat{x} \) and \( \hat{y} \) 
can be connected by a magnetic geodesic \( \gamma \) 
of \( (\hat{G}, \hat{\mathcal{G}}, \hat{\sigma}) \) with energy \( \kappa \). 
Since, by definition, the projection 
\( \pi \colon \hat{G} \to G \) maps magnetic geodesics of 
\( (\hat{G}, \hat{\mathcal{G}}, \hat{\sigma}) \) with energy \( \kappa \) 
to magnetic geodesics of \( (G, \mathcal{G}, \sigma) \) with the same energy, 
it follows that the projected curve \( \pi(\gamma) \) 
is a magnetic geodesic of \( (G, \mathcal{G}, \sigma) \) with energy \( \kappa \) 
connecting \( x \) and \( y \). 
This completes the proof of~\ref{it: main thm e int}.

Moreover, magnetic geodesic completeness for the magnetic systems 
\( (G^\ell, \mathcal{G}, \sigma) \) 
follows directly from the no-loss–no-gain result 
(\Cref{cor: no-loss-no-gain for exp}).
\end{proof}

\appendix
\section{The existence of a Chern connection}\label{appendix:chern_connection}

This appendix is devoted to proving the existence of a Chern connection for the Finsler manifold \( (\hat{G}, \mathcal{F}^\kappa) \), as stated in \Cref{prop: conjugated to Finsler flow_sec}.  
In what follows, we use the explicit form of \( \mathcal{F}_{\kappa} \) to show that, in this specific situation, the argument from the finite-dimensional case carries over.

We have already seen that the function \( \mathcal{F}^\kappa \) in~\eqref{e:finsler_metric_sec}
defines a \( C^\infty \) Finsler structure on \( TM \setminus \{0\} \) that is \emph{strongly convex}.  
That is, for each \( (x,v) \in TM \setminus \{0\} \), the fundamental tensor
\[
g^{\mathcal{F}^\kappa}_{(x,v)}(u,w)
:= \frac{1}{2}\,\partial_v^2\big(\mathcal{F}^\kappa(x,v)^2\big)[u,w]
\]
is given, by a standard computation analogous to the finite-dimensional case, by
\begin{align*}
g^{\mathcal{F}^\kappa}_{(x,v)}(u,w)
&= 2\kappa\,\frac{g_x(u,w)}{\Vert v \Vert_{g}}
   - \frac{\sqrt{2\kappa}}{\Vert v \Vert_{g}}
     \big( g_x(v,u)\,\alpha_x(w) + g_x(v,w)\,\alpha_x(u) \big) \\
&\quad + \frac{\alpha_x(v)\,\sqrt{2\kappa}}{\Vert v \Vert_{g}^3}\, g_x(v,u)\,g_x(v,w),
\end{align*}
which defines a continuous, symmetric, positive-definite bilinear form on \( T_xM \).  
It depends \( C^\infty \)-smoothly on \( (x,v) \) and induces a topological isomorphism
\[
g^{\mathcal{F}^\kappa\,\vee}_{(x,v)} : T_xM \to T_x^*M.
\]
This tensor is symmetric and positive definite provided that
\( \Vert \alpha \Vert_{\infty} < \sqrt{2\kappa} \),
which guarantees the strong convexity of \( \mathcal{F}^\kappa \).

Consequently, the \emph{Legendre transform} of the Lagrangian \( \tfrac12(\mathcal{F}^\kappa)^2 \) is
\[
\mathcal{L}(x,v)
:= \Big(x,\, \partial_v\big(\tfrac{1}{2}(\mathcal{F}^\kappa)^2\big)(x,v)\Big)
= (x,\, g^{\mathcal{F}^\kappa}_{(x,v)}(v,\cdot)),
\]
which is a \( C^\infty \) Banach bundle diffeomorphism from \( TM \setminus \{0\} \) onto
\( T^*M \setminus \{0\} \).  
By \cite[§7.3]{MardsenRatiuMechanics}, the Euler–Lagrange equations associated with the regular Lagrangian
\( \tfrac12(\mathcal{F}^\kappa)^2 \)
define a smooth second-order vector field (the \emph{canonical spray}) on \( TM \setminus \{0\} \).
It is well known that any such spray induces a nonlinear connection on \( TM \setminus \{0\} \)
from which the \emph{Chern (Chern–Rund) connection} is canonically obtained.
This connection is the unique torsion-free linear connection on the pullback bundle
\( \pi^*TM \to TM \setminus \{0\} \) that is almost \( g^{\mathcal{F}^\kappa} \)-compatible.

\bibliographystyle{abbrv}
	\bibliography{ref}
\end{document}